\documentclass[12pt]{amsart}

\usepackage{amssymb,enumerate,amsthm,yfonts,mathrsfs, esint}
\usepackage{amsmath}
\usepackage{bbm}           
\usepackage{bm} 
\usepackage{eso-pic,graphicx}
\usepackage{tikz}

\usepackage{mathtools}
\usepackage{rotating}

\usepackage[colorlinks=true, pdfstartview=FitV, linkcolor=blue, citecolor=blue, urlcolor=blue]{hyperref}

\calclayout

\makeatletter
\def\Ddots{\mathinner{\mkern1mu\raise\p@
\vbox{\kern7\p@\hbox{.}}\mkern2mu
\raise4\p@\hbox{.}\mkern2mu\raise7\p@\hbox{.}\mkern1mu}}
\makeatother

\def\XXint#1#2#3{{\setbox0=\hbox{$#1{#2#3}{\int}$}
\vcenter{\hbox{$#2#3$}}\kern-.5\wd0}}

\newtheorem{theorem}{Theorem}[section]

\newtheorem{corollary}[theorem]{Corollary}

\newtheorem{definition}[theorem]{Definition}
\newtheorem{example}[theorem]{Example}

\newtheorem{lemma}[theorem]{Lemma}

\newtheorem{remark}[theorem]{Remark}

\numberwithin{equation}{section}
 \allowdisplaybreaks

\def\R{\mathbb R}

\def\bey{\begin{eqnarray*}}
\def\eey{\end{eqnarray*}}

\def\D{{\mathscr D}}

\newcommand\f[2]{\frac{#1}{#2}}
\newcommand\de{\delta}

\newcommand\ee{\varepsilon}

\title[A new sufficient two-weighted bump assumption for $L^p$ boundedness]{A new sufficient two-weighted bump assumption for $L^p$ boundedness of Calder\'on-Zygmund operators}

\author{Theresa C. Anderson}

\address{Theresa C. Anderson, Department of Mathematics, Brown
  University, Providence, RI 02912, USA}
\email{theresa\_anderson@brown.edu}

\thanks{The author is supported by an National Science Foundation graduate student fellowship.}

\subjclass[2010]{}


\date{March, 27, 2013}

\begin{document}

\begin{abstract}
We present new results on the two-weighted boundedness of singular integral operators and $L^p$ boundedness of the Orlicz maximal function.  Namely, we extend a theorem of P\'erez regarding the necessary and sufficient conditions for the boundedness of the Orlicz maximal function as well as give a new sufficient two-weighted boundedness assumption for Calder\'on-Zygmund singular integrals.   
\end{abstract}
\maketitle

\section{Introduction}
An active area of harmonic analysis lately has been discovery of the precise dependence of the constants in weighted norm inequalities for Calder\'on-Zygmund singular integral operators (CZOs).  The recently solved so-called $A_2$ theorem (\cite{Hy2012}) gives the sharp dependence of the constant on the weight characteristic in one-weighted norm inequalities for CZOs.  Additionally, in searching for a proof for this theorem, analysts developed the new theory of non-homogeneous analysis.  A natural follow-up to the one-weighted results are related questions about two-weighted norm inequalities for CZOs, but in most cases basic bounds are not even known.  For example, necessary and sufficient conditions for $L^p$ boundedness are only known for a few operators, such as the maximal function \cite{sawyer82b} and Hilbert transform \cite{laceysawyerhilbert}, \cite{2013arXiv1301.4663L}.  The two-weighted theory is a developing field, with relatively little known compared to the one-weighted theory.  We hope that the background and results discussed further propel development in these areas.   

 We seek to determine sufficient conditions on the weights $(w,v)$ that give us a two-weighted bound for most CZOs, that is, a bound of the form: \[\|Tf\|_{L^p(w)}\leq C\|f\|_{L^p(v)}.\]  The bump condition theory that we will discuss was introduced to help answer this question.  We will relate some history of determining necessary and/or sufficient conditions for $L^p$ boundedness in the classical case of $\R^n$ with Lebesgue measure.  We assume $1<p<\infty$ unless otherwise stated throughout this article.

An initial idea to find sufficient conditions for $L^p$ boundedness in the two-weight case was to form a two-weighted analogue of the Muckenhoupt $A_p$ condition, that is, for a pair of weights $(w,v)$:
\begin{equation}
\label{twoweightAp}
\sup_Q\fint_Qwd\mu\left( \fint_Q v^{-p'/p}d\mu\right)^{p-1}<\infty,
\end{equation}
 where $Q$ is a cube and we use the notation $\fint_Q = \frac{1}{\mu(Q)}\int_Q$.  However, \ref{twoweightAp} was only necessary and sufficient for the maximal function to be bounded on weak $L^p$, and only necessary for boundedness of some "interesting operators" under consideration (see \cite{2012arXiv1202.2406N} and  \cite{sawyer82b} for a discussion of this).  Therefore, other conditions were then proposed on $(w,v)$ to give sufficient boundedness conditions for CZOs.  Neugebauer helped to pioneer this effort \cite{MR687633}, by "bumping up" the power on the weights into a condition like:
 \begin{equation}
 \label{Neubump}
 \sup_Q\left( \fint_Qw^{pr}d\mu\right) ^{1/pr}\left( \fint_Q v^{-p'r}d\mu\right)^{1/p'r}<\infty,
 \end{equation}
 for $r>1$.  This condition was strong enough to ensure that both the maximal function and CZOs were bounded from $L^p(v^p)$ to $L^p(w^p)$.  A further generalization of \ref{twoweightAp} came from Orlicz space theory.  Orlicz norms provide a finer scale of norms ''in between" the $L^p$ ones, which we will make precise later.  Andrei Lerner used Orlicz $B_p$ norms on both weights in his assumption to prove sufficiency for strong $L^p$ boundedness of all CZOs, answering a question of Cruz-Uribe and P\'erez \cite{Lern2012}.
 
 This paper contains two main results.
 The first is a \emph{Coifman-Fefferman inequality}: bounding a singular integral operator by a corresponding maximal function, which in many cases, leads to a two-weighted bound for CZOs.  Our assumptions are related to those in the still open ''separated bump conjecture" (see \cite{MR2797562}, \cite{CRV2012}, \cite{ACM}, \cite{NRV}), but not exactly the same since they represent a hybrid of both Neugebauer's and Lerner's assumptions.  The proof is a twist on an extrapolation technique from one-weighted theory.  As far as we are aware, the use of the reverse H\"older extrapolation in the two-weighted theory is new.

To state this result, we recall the sparse operators of Lerner \cite{lerner-IMRN2012}, \[T^Sf = \sum_{Q\in S}(\fint_Q f)\chi_Q,\] where $Q$ are dyadic cubes and $S$ represents a sparse family: a collection of disjoint dyadic cubes such that for $Q\in S$, \[\mu\left( \bigcup_{Q'\subsetneq Q, Q'\in S}Q'\right)\leq \frac{\mu(Q)}{2},\] and mention that in \cite{lerner-IMRN2012}, Lerner proved that every CZO can be bounded above in norm by a supremun of these sparse operators.  The theorem of Lerner has been proven in the case of spaces of homogeneous type in \cite{ACM} and \cite{AV-2012}.
    \begin{theorem}
    \label{mainextrapolation}
    Let $Y$ be a Banach function space (if $Y = L^B$ this is the Orlicz norm), $T^S$ be a sparse operator on $\R^n$ with Lebesgue measure and $1<p<\infty$.  Let $(w,v)$ be weights such that  $\sup_Q\|w\|_{q,Q}\|v^{-1}\|_{Y,Q}\leq K$ for some $q >p$ and all $Q\in S$.  Then 
  \begin{equation}
  \label{TsandMy'}
  \int_{\R^n}(T^Sfw)^pdy\leq C^p\int_{\R^n}M_{Y'}(fv)(y)^pdy,
  \end{equation}
  where $M_{Y'}$ is a generalization of the Orlicz maximal function.
  \end{theorem}
  If the maximal function $M_{Y'}$ is bounded on weighted $L^p$, then we obtain a two-weighted $L^p$ bound for CZOs.  
  
  The proof of this theorem is general enough that it extends to spaces of homogeneous type.
  
  Secondly, we prove that the Orlicz maximal function, an extension of the usual Hardy-Littlewood maximal function, is bounded on $L^p$ in spaces of homogeneous type without assuming that the associated Young function is doubling.  This generalizes a result of P\'erez \cite{perez95}.  The importance of the Orlicz maximal function is underscored by its appearance in the proof of the first main result.  Moreover, this maximal function has become an integral part of the bump theory and the search for sufficient conditions for two-weighted bounds.
    
  This paper is organized in the following fashion.  We introduce some background in section 2 and mention a few preliminary results in section 3, including the result on the Orlicz maximal function.  Section 4 contains some background for the two-weighted theorem, proved in section 5.  In section 6, we give some corollaries, including a strong two-weighted bound for Calder\'on-Zygmund operators, and examples.    
  
  For notation, we use the standard $H\lesssim I$ to mean $H\leq CI$, where $C$ is a constant.  We typically use the letters $C$ and $K$ to denote constants, even if they change from line to line.
  
  \section{Background}
  In \cite{perez95}, Carlos P\'erez defined a property of Young functions termed ``the $B_p$ condition", which is important to boundedness of different types of maximal functions, important in weighted theory.
  
  \begin{definition}
  \label{Young1}
  A Young function is a a convex, increasing, continuous function $A :[0,\infty)\to[0,\infty)$ such that $A(0) = 0$ (see \cite{perez94}).
  \end{definition}
We sometimes normalize $A(1) = 1$.
  
  \begin{definition}
  Let $A$ be a Young function.  The complementary Young function $\bar{A}$ is defined as:
  \[\bar{A}(s) = \sup_t\{st-A(t)\},\] (see \cite{MR2797562}).
  \end{definition}  
  
  We often work with spaces of homogeneous type, which occur frequently in applications.
  A space of homogeneous type  (abbreviated SHT) is a triple
  $(X, \rho, |\cdot|)$ where $X$ is a set, $\rho$ is a quasimetric, and the positive measure $\mu$ is \emph{doubling}, that is $$0<\mu(B(x_0,2r))\leq C_d\mu(B(x_0,r))<\infty.$$
  
  Now, we can use a Young function to define the Orlicz norm.
  \begin{definition}
  Given a Young function $A$ and a set $E$ such that $\mu(E)>0$, define the Orlicz space norm of a function $v\in L^1_{loc}$,  \[\|v\|_{A,E} = inf\{\lambda>0:\fint_EA\left( \frac{|v(x)|}{\lambda}\right) d\mu \leq 1\}.\]
  \end{definition}
  \begin{definition}
  The corresponding Orlicz maximal function is
  \[M_A(v(x)) = \sup_{B\ni x}\|v\|_{A,B},\] where $B$ ranges over balls.
  \end{definition}
  Note that sometimes we define this maximal function with respect to cubes or to dyadic cubes.
  
  \begin{remark}
  Note that if $A = t^p$, this norm becomes the normalized $L^p$ norm, that is: $\|f\|_{p,Q} = \left( \fint_Q f^p\right) ^{1/p}.$
  \end{remark}

All Young functions satisfy \emph{Young's inequality}, that is \[ab\leq A(a)+\bar{A}(b).\]   We also have the \emph{generalized H\"older's inequality}: \[\int_Xfg\leq \|f\|_A\|g\|_{\bar{A}}.\]

  \begin{definition}
  A Young function $A(t)$ is doubling if \[A(2t)\leq CA(t)\] for all $t$ and for a fixed constant $C$.
  \end{definition}  
  
  Some examples of Young functions include power functions $t^s$, log bumps \[\frac{t^p}{\ln^{1+\de}(1+t)}\] and loglog bumps \[\frac{t^p}{\ln(1+t)(\ln\ln(1+t))^{1+\de}}.\]
  
  \section{$B_p$ classes and maximal function results}
 The following class of Young functions is useful in both boundedness of associated maximal functions and in two-weight norm inequalities.
  
  \begin{definition}
  \label{Bp1}
  A Young function $A$ belongs to $B_p$ (we write $A\in B_p$) for some $1<p<\infty$ if 
  \[
  \int_c^{\infty}\frac{A(t)}{t^p}\frac{dt}{t}<\infty,\]
  for some $c>0$.
  \end{definition}

  It would be desirable to determine when the Orlicz maximal function is bounded.  We start by mentioning the $L^{\infty}$ bound, whose proof is included for completeness.
  \begin{lemma}
  If $A$ is a Young function then $\|M_Af\|_{L^{\infty}}\lesssim \|f\|_{L^{\infty}}$.
  \end{lemma}
  \begin{proof}
  Let $Q$ be a cube and let $\lambda = \|f\|_{\infty}$, then   $\frac{|f(y)|}{\lambda}\leq 1$ a.e on $Q$.  Since $A$ is increasing and $A(1)=1$, it follows that \[A\left( \frac{|f(y)|}{\lambda}\right) \leq 1\] a.e. on $Q$.  Hence \[\fint_Q A\left( \frac{|f(y)|}{\lambda}\right) \leq 1.\]  Therefore \[\inf_{\lambda >0}\{\lambda : \fint_Q A\left( \frac{|f(y)|}{\lambda}\right) \leq 1\}\leq \|f\|_{\infty}.\]  This is true for all $Q$, so by taking suprenums, the result follows.
  \end{proof}
  
  Many $L^p$ results concerning the maximal function can be found in \cite{perez94}, \cite{perez95}, \cite{perezwheeden}.  One of the most useful criteria is an $L^p$ boundedness result on $\R^n$ explicitly tied to the $B_p$ condition.
  
  \begin{theorem} 
  \label{Bpimplies}
   \cite{perez94} Let $1<p<\infty$ and $A$ be a doubling Young function with $\frac{A(t)}{t}$ increasing, $A(t) \to \infty$ as $t\to \infty$, and $A(1) = 1$.  Then the following are equivalent:
  \begin{enumerate}
  \item{$A\in B_p$}
  \item{There is a constant $c$ such that $\|M_Af\|_p\leq c\|f\|_p$ for all nonnegative $f$.}
  \end{enumerate}
  \end{theorem}
    Though doubling was originally assumed, it is actually not needed (see below \ref{regularityBp}).
    
   If $(X,\rho,\mu)$ is an SHT and $\mu(X) = \infty$ then this theorem is still true \cite{MR2022366}.
   However, if $\mu(X) < \infty$, this theorem fails (see the counterexample mentioned in \cite{MR2022366}).
 
  Inspired by a question from Cruz-Uribe and recent result from \cite{2012arXiv1211.2603L}, we present the following which extends \ref{Bpimplies} to the setting of infinite measure spaces of homogeneous type, without requiring the doubling of the Young function.
  
  \begin{theorem}
  \label{regularityBp}
Let $(X,\rho,\mu)$ be a space of homogeneous type with doubling constant $C_d$ and $\mu(X)=\infty$.  Let the assumptions of \ref{Bpimplies} hold.  
Then $M_A$ is bounded on $L^p$ if and only if $A\in B_p$.
  \end{theorem}
  \begin{proof}
 
  Since the same argument of \ref{Bpimplies} applies in showing that if $A\in B_p$ then $M_A$ is bounded, we will only show the converse. 
 Assume that $M_A$ is bounded.  Then using the definition, one can show for any measurable set $B$ that  \[M_A(\chi_B)(x) = \sup_{B'\ni x}\frac{1}{A^{-1}(\frac{\mu(B')}{\mu(B'\cap B)})}\] where $B'$ are balls (see for example, 
    \cite{perezwheeden}).  Note that if $x\in B$, $M_A(\chi_B)(x) = 1$ and for $x\notin B$, \[1\geq M_A(\chi_B)(x)>\frac{1}{A^{-1}\left( \frac{\mu(B')}{\mu(B)}\right) },\] where $B'\ni x$.
 Assume for the moment that $X$ is nonatomic and fix $x_0\in X$.  Later we will lift the nonatomic assumption.  Define $r := r_x = \rho(x,x_0)$ for any $x\in X$.  Let
 \begin{equation}
 \label{SHTradius}
 F(r) = \mu(B(x_0,r))
 \end{equation} 
 Clearly $F$ is non-decreasing.  Since $\mu(X) = \infty$, we can define an inverse for $F$ as   
 
  \begin{displaymath}
    F^{-1}(s) = \left\{
      \begin{array}{lr}
       \inf\{r: F(r) = s\} & : F(r) = s \text{ for some } r\\
        \sup\{r:F(r)<s\} & \text{otherwise}
      \end{array}
    \right.
 \end{displaymath}

 Note that since $x_0$ is not an atom, then every $s$ has a well-defined inverse.
 Hence, if $B' = B(x_0,r)$, we have that \[\mu\left( x\in  X: M_A(\chi_B)(x)>t\right) \geq \mu\left( x\in  X: \frac{1}{A^{-1}\left( \frac{\mu(B')}{\mu(B)}\right)}>t\right) = \mu\left( x\in  X: \frac{1}{A^{-1}\left( \frac{F(r)}{\mu(B)}\right) }>t\right)\] 
  by \ref{SHTradius}.
 
 Denote $\mu(B) = b$.  
 Since $M_A$ is bounded on $L^p$, then surely it is bounded for the function $f = \chi_B$, so using the above estimate:
 \[ \infty >\int_XM_A(\chi_B)(x)^pd\mu \geq p\int_0^{\infty}t^p\mu\left( x\in  X: \frac{1}{A^{-1}\left( \frac{F(r)}{b}\right) } >t\right) \frac{dt}{t}\] \[
 \geq p\int_0^{\infty}t^p\mu\left(  A\left(\frac{1}{t}\right)b>F(r)\right) \frac{dt}{t}
 \geq p\int_0^{\infty}t^p\mu\left(   F^{-1}A\left(\frac{1}{t}\right)b>r\right) \frac{dt}{t}\]
 since $\{r: F^{-1}(s)>r\}\subseteq \{r:s>F(r)\}$.  Continuing, the above
 \[= p\int_0^{\infty}t^pFF^{-1}\left( A\left( \frac{1}{t}\right)b\right) \frac{dt}{t} \geq p\int_0^{\infty}t^pA\left( \frac{1}{t}\right)\frac{b}{C_d}\frac{dt}{t}. \] 
 Here we have used an important fact from the definition of $F^{-1}$: if $F(r_1) = FF^{-1}(s)$ then \[\frac{1}{C_d}s\leq \frac{1}{C_d}F(r_1+\ee)\leq F\left( \frac{r_1+\ee}{2}\right) \leq F(r_1) = FF^{-1}(s),\] for some $\ee>0$ such that $F(\frac{r_1+\ee}{2})\leq F(r_1)$.  This is where doubling of the measure $\mu$ is used.  For many $s$, we have $FF^{-1}(s) = s$ and do not need to use the above fact.  However, for certain $s$ that fall in the range on the $s$-axis where $F$ has a jump discontinuity; doubling allows us to replace $FF^{-1}\left( A(\frac{1}{t})b\right) $ by $A\left( \frac{1}{t}\right)\frac{b}{C_d}$.  
 
 Call $\frac{b}{C_d} = C'$.
 We can finish the argument by recalling that
 \[
 \infty > C'p\int_0^{\infty}t^{p-1}A\left( \frac{1}{t}\right) dt> C'p\int_0^{t_0}t^{p-1}A\left( \frac{1}{t}\right) dt = C'p\int_{1/t_0}^{\infty}\frac{A(t)}{t^p}\frac{dt}{t}.\]

 Thus we have actually found that \[\frac{p}{C_d}\int_{1/t_0}^{\infty}\frac{A(t)}{t^p}\frac{dt}{t}\leq \|M_A\|_{L^p}^p\] since $\|\chi_B\|_{L^p}^p = b$.

 Now, if $x_0$ is an atom, then $\mu(x_0) = \mu(B(x_0,\ee)) = \de$ for some $\ee$, $\de >0$, so $F^{-1}(s)$ is defined for $s\geq\de$.  For atoms, the only issue that arises in the proof is when we take $F^{-1}\left( A\left( \frac{1}{t}\right) b\right) $, so as long as $A\left( \frac{1}{t}\right)\geq\frac{\de}{b}$, no adjustment to our argument is needed.  
 
 In other words, $F^{-1}\left( A\left( \frac{1}{t}\right)b\right)$ only fails to be defined when $A\left( \frac{1}{t}\right)<\frac{\de}{b}$.
 Therefore, if $\de \leq b$, $F^{-1}\left( A\left( \frac{1}{t}\right)b\right)$ can only fail to be defined when $t> 1$ since $A$ is increasing and $A(1) = 1$.  
 
 However, for $t> 1$, $\mu(x\in X: M_A(\chi_B)(x)>t) = 0$.  We only need to consider level sets where $t< 1$ in our argument, and $F^{-1}$ exists for all points in these sets. 
 
 Since from the beginning, we were free to choose any measurable $B$, let $B = B(x_0,\ee)$ for some atom $x_0$.  Then \[b = \mu(B) = \mu(B(x_0,\ee)) = \de,\] as wanted.  Therefore, the above proof holds for spaces where every point is an atom as well.

 \end{proof}
  
 \begin{remark}
 If $\mu(X)<\infty$, then there exists some $R$ such that $B(x_0,R) = X$.  Then in the above proof, $\sup\{r:F(r)<F(R) +1\} = \infty$, so $F^{-1}$ would not be well-defined.  Since $\mu(X) = \infty$, then for all $s$ there exists $r$ such that $F(r)\geq s$.
 
 Note that we could also have defined $F^{-1}$ using $\inf\{r:F(r)>s\}$ instead of $\sup\{r:F(r)<s\}$.  This definition would have led to the same result.
 
 \end{remark}
  
We now give some background so we can state an earlier theorem of P\'erez about maximal function bounds.
 Let $Y$ be a Banach function space normed in the following manner: $\|v\|_{Y,Q} = \sup\{\fint_Q |fg|d\mu :g\in Y', \|g\|_{Y'}\leq 1\}$, where $Y'$ is the associate space: \[Y' = \left\lbrace f \text{ measurable} :\|f\|_{Y'}:=\sup_{\|g\|_Y \leq 1}\int_X|fg|d\mu<\infty\right\rbrace.\]  Banach function spaces include Orlicz spaces and $L^p$ spaces.  For example, if $Y = L^A$ with $A$ a Young function, then $\|\cdot\|_Y = \|\cdot\|_A$ and $Y' = L^{\bar{A}}$.  For more information, see \cite{perez95} or \cite{lernerpositivedyadic}.  Let \[M_{Y'}f(x) = \sup_{Q\ni x}\|f(x)\|_{Y',Q}\] for cubes Q, be the associated maximal function.

  Now we can state the theorem of P\'erez \cite{perez95}.
  \begin{theorem}
  \label{onebumpimpliesM}
  Assume $M_{Y'}$ is bounded on $L^p(\R ^n)$ where $Y$ is a Banach function space and let $(w,v)$ be a pair of weights such that \begin{equation}
  \label{onebump}
  \sup_Q\|w\|_{p,Q}\|v^{-1}\|_{Y,Q}\leq K.
  \end{equation}
  Then \[\int_ {\R^n}(M_{Y'}fw)^pdx\leq c\int_{\R^n}(fv)^pdx\] for all nonnegative $f$. 
  \end{theorem} 
  
  This theorem is also true in an SHT (see \cite{perezwheeden}).

  \section{Preliminary weighted results}
 To prove our two-weighted extrapolation result, we need the Reverse H\"older classes (\cite{CUN}).
  \begin{definition}
  Fix $1<p<\infty$.  We say that a weight $\rho$ belongs to the Reverse H\"older class $p$, denoted $RH_p$, if for all cubes $Q$,
  \begin{equation} 
  \label{revholder}
  \left( \fint_Q\rho^pd\mu\right) ^{1/p}\leq \kappa \f{\rho(Q)}{\mu(Q)}.
  \end{equation}
  The smallest such $\kappa$ is called $[\rho]_{RH_p}$.
  \end{definition}
  Though general cubes are not well-defined in an SHT, we have a dyadic cube construction (see \cite{MR1104656}, \cite{MR2901199}). In our main theorem, we only work with dyadic cubes, so we can take \ref{revholder} with respect to dyadic cubes.  We remark that all $A_{\infty}$ weights satisfy the reverse H\"older property.
    
  We recall the definition of sparse families and operators, intimately related to CZOs by recent work of Lerner \cite{Lern2012}.  Sparse families generalize the Calder\'on-Zygmund decomposition.
  \begin{definition}
  A sparse family $S = \cup_kS_k$, $S_k\in \D_k$ on $\D$ is a collection of dyadic cubes such that for $Q\in S$, \[\mu\left( \bigcup_{Q'\subsetneq Q, Q'\in S}Q'\right) \leq \frac{\mu(Q)}{2}.\] 
  
  \end{definition}
  A way to construct a disjoint family from a sparse one, where the measure of each cube in $S$ is approximately the same as a disjoint partner, is the following:
  \begin{definition}
  Let \[E(Q) = Q\setminus\bigcup_{Q'\subsetneq Q, Q'\in S}Q'.\]
  \end{definition}
  It is easy to see using the definitions that $E(Q)$ is disjoint and that \[\mu(E(Q))\leq \mu(Q)\leq 2\mu(E(Q)).\]
  
  \begin{definition}
  A sparse operator is a simple averaging operator over a select set of cubes, that is:
   \[T^S(f) = \sum_{Q\in S}(\fint_Qf)\cdot \chi_Q.\]
  \end{definition}
 
  \begin{definition}\label{calderon}
  We'll say that $K:X\times X\setminus\lbrace{x=y\rbrace}\to R$ is a Calder\'on-Zygmund kernel if there exist $\eta>0$ and $C<\infty$ such that for all $x_0\neq y\in X$ and $x\in X$ it satisfies the decay condition:
  \begin{equation}
  \label{decay}
  |K(x_0,y)|\leq \frac{C}{|B(x_0,\rho(x_0,y))|}
  \end{equation}
  and the smoothness condition for $\rho(x_0,x)\leq \eta\rho(x_0,y)$:
  \begin{equation}
  \label{smoothness}
  |K(x,y)-K(x_0,y)|\leq\left(\frac{\rho(x,x_0)}{\rho(x_0,y)}\right)^{\eta}\frac{1}{|B(x_0,\rho(x_0,y))|},
  \end{equation}
  \begin{equation*}
  |K(y,x)-K(y,x_0)|\leq\left(\frac{\rho(x,x_0)}{\rho(x_0,y)}\right)^{\eta}\frac{1}{|B(x_0,\rho(x_0,y))|}.
  \end{equation*}
  \end{definition}
  
  If in addition, the associated operator $T$ is bounded on $L^2$, we call it a \emph{Calder\'on-Zygmund operator}.
  
  \begin{lemma}
  \label{sparsefamcriterionRH}
  If $w\in A_{\infty}$, then
   \[w(Q)\leq \gamma w(E(Q))\] where $\gamma$ depends on $X$ and on $[w]_{A_{\infty}}$ only. 
  \end{lemma}
  \begin{proof}
  Let $\bigcup_{Q'\subsetneq Q, Q'\in S}Q' = T$.  Due to the condition of sparse family, we have 
  \[\mu(T)\leq \frac{\mu(Q)}{2}\] so by the definition of $A_{\infty}$ we have that \[w(T)\leq \beta w(Q)\] for some $0<\beta<1$.  Thus, $w(Q) = w(T)+w(E(Q)) \leq \beta w(Q)+w(E(Q))$ so $(1-\beta)w(Q)\leq w(E(Q))$ or \[w(Q)\leq \frac{1}{1-\beta}w(E(Q)).\] 
  \end{proof}
  
  We now come to the proof of \ref{mainextrapolation}.

   \section{Proof of two-weighted theorem}
   \begin{proof}
   First, in Step 1, we show for $p=1$ and any $\rho\in RH_{q'}$ where $q$ is in the above range that \[\int_XT^Sfw\rho d\mu(y)\leq C\int_XM_{Y'}(fv)(y)\rho d\mu(y).\]  We then move to the main statement via an extrapolation argument in Step 2.
   
   \emph{Step 1}
   For this step, we use the notation $w\rho(Q) = \int_Qw\rho d\mu$.  Then
   \[\int_XT^Sfw\rho d\mu(y) = \sum_{Q\in S}\left( \fint_Qf\right) w\rho (Q) = \sum_{Q\in S}\left( \fint_Qf\cdot v\cdot v^{-1}\right) w\rho (Q)\] \[\leq \sum_{Q\in S}\|fv\|_{Y',Q}\|v^{-1}\|_{Y,Q}w\rho (Q) = \sum_{Q\in S}\|fv\|_{Y',Q}\|v^{-1}\|_{Y,Q}\frac{w\rho (Q)}{\rho (Q)}\rho (Q) := I.\]  Now, estimating part of this expression, we get \[\|v^{-1}\|_{Y,Q}\frac{w\rho (Q)}{\rho (Q)} \leq \|v^{-1}\|_{Y,Q}\left( \int_Qw^qd\mu\right) ^{1/q}\left( \int_Q\rho ^{q'}d\mu\right) ^{1/q'}\cdot \f{1}{\rho (Q)}\] \[= \|v^{-1}\|_{Y,Q}\left( \fint_Qw^qd\mu\right) ^{1/q}\left( \fint_Q\rho ^{q'}d\mu\right) ^{1/q'}\f{\mu (Q)}{\rho (Q)} \leq K_q\kappa_q,\] by \ref{onebump} and \ref{revholder}.
   Then \[I\leq K_q\kappa_q \sum_{Q\in S}\|fv\|_{Y',Q}\rho (Q) \leq \gamma K_q\kappa_q\sum_{Q\in S}\|fv\|_{Y',Q}\rho (E(Q))\] due to \ref{sparsefamcriterionRH}.  The above is bounded by \[C\sum_{Q\in S}\int_{E(Q))}M_{Y'}(fv)d\rho(y)\leq C\int_XM_{Y'}(fv)\rho d\mu(y),\] where $C = \gamma K_q\kappa_q$.
   This gives Step 1.
   
   \emph{Step 2:}
   For extrapolation, we'll follow the procedure of Rubio de Francia.  However, we need to adapt this to Reverse H\"older classes since Step 1 is a statement for all $\rho \in RH_{q'}$.
   
   Define the Rubio de Francia operator \[Rh = \sum_{k=0}^{\infty}\f{M^kh}{2^k\|M\|_{L^{p'/q'}}^k}.\]
   Let \[\tilde{R}(h) = R(h^{q'})^{1/q'}.\]
   Note that for $h\in L^{p'}(d\mu)$,
   \begin{enumerate}
   \item{$h\leq \tilde{R}(h)$}
   \item{$\|\tilde{R}(h)\|_{p'}\leq 2^{1/q'}\|h\|_{p'}$}
   \item{$M(R(h^{q'}))\leq 2\|M\|_{L^{p'/q'}}R(h^{q'})$}
   \end{enumerate}
   Item 1 is clear since all terms are positive.  Item 2 can be seen by the following argument: \[\|\tilde{R}(h)\|_{p'} = \|R(h^{q'})\|_{p'/q'}^{1/q'} = \left( \|\sum_k\f{M^k(h^{q'})}{2^k\|M\|_{p'/q'}^k}\|_{p'/q'}\right) ^{1/q'}\leq\] \[\left( \sum_k\f{\|M^k(h^{q'})\|_{p'/q'}}{2^k\|M\|_{p'/q'}^k}\right) ^{1/q'}\leq  \left( \sum_k\f{1}{(2\|M\|_{p'/q'})^k}\|M\|_{p'/q'}^k\|h^{q'}\|_{p'/q'}\right) ^{1/q'}\] \[= \|h^{q'}\|_{p'/q'}^{1/q'}2^{1/q'} = 2^{1/q'}\|h\|_{p'}\]
   
   Item 3 implies that $R(h^{q'})=\tilde{R}(h)^{q'}\in A_1$, and since by a similar calculation, we also have $\tilde{R}(h)\in A_1$, then $\tilde{R}(h)\in RH_{q'}$.  (All of these norms are with respect to d$\mu.$)  Notice that we need $p'>q'$ (equivalently $p<q$) for the maximal function norm to be bounded.  
   
   Then we run the Rubio de Francia machine:  by duality,
   \[\|T^Sfw\|_p = \int_XT^Sfw\cdot hd\mu\] for some $h\in L^{p'}(d\mu)$ with norm 1.  Now by property 1 of $\tilde{R}h$, followed by the fact that $R(h^{q'})^{1/q'}\in RH_{q'}$,\[\|T^Sfw\|_p\leq \int_XT^Sfw\cdot \tilde{R}(h)d\mu \leq C\int_XM_{Y'}(fv)\tilde{R}(h)\leq\] \[ C\left( \int_XM_{Y'}(fv)^pd\mu\right) ^{1/p}\left( \int_X(\tilde{R}h)^{p'}d\mu\right) ^{1/p'}\leq C\left( \int_XM_{Y'}(fv)^pd\mu\right) ^{1/p}\left( \int_Xh^{p'}d\mu\right) ^{1/p'}\] \[= C(\int_XM_{Y'}(fv)^pd\mu)^{1/p}\] 
   Taking $p$th powers, we get \ref{TsandMy'} as wanted.
   
   Tracing the constant from Step 1, we see that $C = 2^{1/q'}K_q\cdot[\tilde{R}(h)]_{RH_{q'}} \cdot \gamma,$ where $\gamma$ depends on $[\tilde{R}h]_{A_\infty}, SHT$.  Since we can bound $[\tilde{R}(h)]_{RH_{q'}}\leq [\tilde{R}(h)]_{A_1}\leq (2\|M\|_{p'/q'})^{1/q'}$ we have that $C\leq    2^{2/q'}K_q \|M\|_{p'/q'}^{1/q'} \cdot \gamma$.
   
   \end{proof}

  \section{Corollaries and examples}
  \subsection{Corollaries to \ref{mainextrapolation}}
  \begin{corollary}
  \label{Lernercorollary}
  Let the conditions of \ref{mainextrapolation} hold.  Then \[\|Tf\|_{L^p(w^p)}\leq C\|M_{Y'}(fv)\|_{p}.\]
  \end{corollary}
  \begin{proof}
  By a decomposition of Andrei Lerner \cite{lerner-IMRN2012}, proved in the SHT case in \cite{AV-2012}, we have 
  \begin{equation}
  \|T(f)\|_Z\leq C(SHT, T)\sup_{\D,S} \|T^S(f)\|_Z
  \end{equation}
  where $\D$ is a dyadic grid and $Z$ is a Banach function space, such as $L^p(u)$ (note that $Z$ is not necessarily the same as the Banach function space $Y$ in the theorem).  Recall that $L^p(u)$ is a Banach function space whenever $ud\mu$ is a locally integrable positive function, that is, when $u$ is any weight.  There are a finite number of distinct dyadic grids in any SHT (\cite{MR2901199}).
  \end{proof}
  
  \begin{corollary}
  \label{MYcorollary}
  Let $M_{Y'}$ be bounded on $L^p(\mu)$, and the conditions of \ref{mainextrapolation} hold.  Then,
  \[\|Tf\|_{L^p(w^p)}\leq C\|f\|_{L^p(v^p)}.\]
  \end{corollary}
 
  \begin{remark}
  We can apply \ref{mainextrapolation} to the case $v = w = u^p$ and $Y' = L^p$.  Then our assumption becomes the $A_p$ condition.  However, \ref{MYcorollary} does not hold since $M_{Y'} = M_p$ is not bounded on $L^p$. 
  \end{remark}
  
  \begin{remark}
  Note also that the endpoint case of $q=p$ fails only due to the blowup of the maximal function on $L^{p'/q'}$.  This might give insight on how to prove or disprove the separated bump conjecture, whose assumptions involve two such endpoint statements (see \cite{ACM}).
  \end{remark}
 
  \begin{corollary}
  Let $A$ be a Young function such that $\bar{A}\in B_p$ and the conditions of \ref{TsandMy'} hold.  Then \[\|Tf\|_{L^p(w^p)}\leq C\|f\|_{L^p(v^p)}.\]
  \end{corollary}
  \begin{proof}
  In this case, $Y' = L^{\bar{A}}$, and $M_{\bar{A}}$ is bounded by \ref{Bpimplies}.
  \end{proof} 
  
  Step 2 can be written as a general extrapolation theorem for Reverse H\"older classes:
  \begin{theorem}
  Fix $1<p<\infty$.  Let $S_1(x)$ and $S_2(x)$ be two objects (operators, functions, etc) such that 
  \[\int_XS_1(x)\rho d\mu \leq \int_X S_2(x)\rho d\mu\]
  for all $x\in X$ and for all $\rho \in RH_{q'}$, $q\in (p,p+\ee)$ or $q\in (p,\infty)$.  Then we have that 
  \[\int_X(S_1(x))^p d\mu \leq \int_X (S_2(x))^p d\mu.\]  
  \end{theorem}
  \begin{proof}
  Extrapolation.
  \end{proof}
  
  \subsection{Examples}
  Here we discuss some examples of weights that satisfy the conditions of the extrapolation theorem.  
  
  It is interesting to note the following about power weights.
  \begin{remark}
  Let $w = |x|^{-\alpha}$ and $v = |x|^{-\beta}$ in $\R$.  Then \[\|w\|_{p,Q}\|v^{-1}\|_{p',Q}\leq K\] if and only if $w = v$.  This is due to the stringent requirements on the exponents when we integrate over a cube containing the origin:
  Let $X = \R$ and $Q = [-c,c]$ where $c = \f{\l(Q)}{2}$.  Then (as long as $\alpha\cdot p\neq 1$), \[\left( \fint_Qw^p \right) ^{1/p} = \left( \f{|x|^{-\alpha\cdot p+1}}{D\cdot\l(Q)}\mid ^c_{-c}\right)^{1/p} = D'\l(Q)^{-\alpha}.\]  Similarly, \[\left( \fint_Qv^{-p'}\right) ^{1/p'} = \left( \f{|x|^{\beta\cdot p'+1}}{D\l(Q)}\mid ^c_{-c} \right)^{1/p'} = D'\l(Q)^{\beta},\] where $D$ and $D'$ are constants.  Hence since there are both arbitrarily small and large cubes centered at the origin, in order to have the suprenum over these cubes finite, we must have that $\alpha = \beta$.  Even if we use the Young function $B(t) = t^r$ for $r<p'$, the calculation above for $\left( \fint_Qv^{-r}\right) ^{1/r}$ is the same for cubes around the origin, and we still get that $\alpha = \beta$.   This eliminates certain power weight examples from our consideration. 
  \end{remark}
  Note that the following examples are variations of the classic two-weight "$A_1$ pair" $(u, M_u)$ \cite{MR2797562}.
  \begin{example}
  \normalfont
  Take $M_q$ to be the $L^q$ normalized Hardy-Littlewood maximal function with respect to (dyadic) cubes. We will show that the pair $(u, M_qu)$, where $u$ is a weight, satisfies our assumptions for \ref{mainextrapolation} with the pairs: $(p, B)$ and $(q, B)$, where $B$ is a Young function.

  We have then the following bound (adapted from \cite{ACM}).  For $q>p$ \[M_{q}u(x) \geq \|u\|_{q, Q}\] for all $x\in Q$ so that \[\|u\|_{q,Q}\|M_q(u)^{-1}\|_{B,Q} \leq \|u\|_{q,Q}\|u\|_{q,Q}^{-1}\|\chi_Q\|_{B,Q} = 1.\]  This is true since $\|\|u\|_{B,Q}^{-1}\|_{p',Q} = \|u\|_{B,Q}^{-1}\|1\|_{p',Q} = \|u\|_{B,Q}^{-1}\|\chi_Q\|_{p',Q} = \|u\|_{B,Q}^{-1}.$  Here we have used the normalization $B(1) = 1$.  We also get that \[\|u\|_{p,Q}\|M_q(u)^{-1}\|_{B,Q} \leq  \|u\|_{q,Q}\|M_q(u)^{-1}\|_{B,Q} = 1\] by Holder's inequality.  Hence, the pair $(u, M_q(u))$  satisfies the conditions of \ref{mainextrapolation}.  
  \end{example}

  \begin{example}
  \normalfont
  Take the pair $(M_Au^{1-p}, w)$ where $A$ is a Young function and $u,w$ are weights.  If we also assume that $\|u\|_{A,Q}^{1-p}\|w^{-1}\|_{B,Q}\leq K$ for all $Q$, then \[\|M_Au^{1-p}\|_{p,Q}\|w^{-1}\|_{B,Q}\leq \|u\|_{A,Q}^{1-p}\|w^{-1}\|_{B,Q} < K.\]  Moreover, if we change the first norm to the $\|\cdot\|_{q,Q}$ norm, the bound remains the same.  Thus $(M_Au^{1-p}, w)$ satisfies \ref{mainextrapolation}.
  \end{example}

  \section{Acknowledgments}
  The author would like to thank the NSF for support with a graduate student fellowship.  Additionally, she extends gratitude to Jill Pipher, David Cruz-Uribe, Cristina Pereyra, Tuomas Hyt\"onen and Francesco Di Plinio for helpful comments and discussions.  Thank you also to Carlos P\'erez for accomodations and many motivating discussions in beautiful Sevilla, Spain.

\bibliographystyle{plain}
\bibliography{Thesispaperform092013}

\end{document}